\documentclass[12pt]{amsart}

\usepackage{amsmath,amssymb,enumerate}
\newtheorem{theorem}{Theorem}[section]
\newtheorem{claim}{Claim}
\newtheorem{lemma}[theorem]{Lemma}

\newtheorem{conjecture}[theorem]{Conjecture}
\theoremstyle{definition}

\newcommand{\eps}{\varepsilon}

\newcommand{\bZ}{\mathbb Z}

\DeclareMathOperator{\PG}{PG}

\DeclareMathOperator{\GF}{GF}

\newcommand{\del}{\!\setminus\!}

\begin{document}
\sloppy

\title[Dense binary matroids]{Odd circuits in dense binary matroids}
\author[Geelen]{Jim Geelen}
\address{Department of Combinatorics and Optimization,
University of Waterloo, Waterloo, Canada}
\thanks{This research was partially supported by a grant from the
Office of Naval Research [N00014-10-1-0851].}

\author[Nelson]{Peter Nelson}
\address{Department of Mathematics, Statistics and Operations Research, Victoria University of Wellington, Wellington, New Zealand}

\date{\today}
\begin{abstract}
We show that, for each real number $\alpha > 0$ and
odd integer $k\ge 5$ there is an integer $c$ such that,
if $M$ is a simple binary matroid
with $|M| \ge \alpha 2^{r(M)}$ and with no $k$-element
circuit, then $M$ has critical number at most $c$.
The result is an easy application of a regularity
lemma for finite abelian groups due to Green. 
\end{abstract}

\maketitle

\section{Introduction}

We prove the following:
\begin{theorem}\label{mainsimple}
For each real number  $\alpha >0$ and odd integer $k \ge 5$,
there exists $c \in \bZ$ such that, if $M$ is a simple binary matroid
$M$ with $|M| \ge \alpha 2^{r(M)}$ and with no $k$-element
circuit, then $M$ has critical number at most $c$.
\end{theorem}

The restriction to excluding {\em odd} circuits from a {\em binary}
matroid here is natural. The geometric density Hales-Jewett
theorem [\ref{fk}] implies that dense 
$\GF(q)$-representable matroids with sufficiently large rank
necessarily contain
arbitrarily large affine geometries over $\GF(q)$, which contain
all even circuits when $q = 2$ and all circuits when $q > 2$.
So dense $k$-circuit free $\GF(q)$-representable matroids of
large rank only exist when $q = 2$ and $k$ is odd. 

Our main theorem (Theorem~\ref{maintech}) is somewhat more
general than Theorem~\ref{mainsimple};
it bounds the critical number of any sufficiently dense binary matroid
whose elements are each contained in at most
$o(2^{(k-2)r})$ circuits of size $k$.
Note that each element of $\PG(r-1,2)$ is contained in at
most $2^{(k-2)r}$ circuits of size $k$, so our result is
best possible up to a constant factor.  We obtain the theorem as an easy
application of Green's regularity lemma for finite
abelian groups~[\ref{green}], which we review in Section 2.

Recall that, if $M$ is a simple rank-$r$ binary matroid
considered as a restriction of the binary
projective geometry $G \cong \PG(r-1,2)$, then the
\emph{critical number}  of $M$ is the minimum
$c \in \bZ_0^+$ such that $G$ has a rank-$(r-c)$
flat disjoint from $E(M)$. Equivalently, the critical number
is the minimum number of ``cocycles" needed to cover $E(M)$,
where by a cocycle we mean a disjoint union of cocircuits.
Thus cocycles correspond to cuts in a graph and, hence,
critical number is a geometric analog of chromatic number.

Theorem~\ref{mainsimple} is analogous to the following theorem
due to Thomassen~[\ref{t07}].

\begin{theorem}\label{thomassen}
For each real number $\alpha >0$ and odd integer $k \ge 5$,
there exists $c \in \bZ$ such that every simple graph on
$n$ vertices with minimum degree at least $\alpha n$ and
no $k$-cycle has chromatic number at most $c$. 
\end{theorem}

Theorem~\ref{thomassen} does not extend to the case that $k=3$;
for each $\eps>0$,
Hajnal (see~[\ref{erdossimonovits}]) gave examples
of triangle-free graphs $G$ with minimum degree at least
$(\frac 1 3 -\eps) |V(G)|$
and with  arbitrarily large chromatic number.
Nevertheless, we conjecture that
Theorem~\ref{mainsimple} also holds for $k = 3$.
That is:

\begin{conjecture}\label{triangleconj}
For each real number $\alpha>0$ there exists $c \in \bZ$ such that,
if $M$ is a  simple triangle-free binary matroid
with $|M| \ge \alpha 2^{r(M)}$, then $M$ has critical number at most $c$.
\end{conjecture}

Green's regularity lemma gives a weaker outcome:

\begin{theorem}\label{weaker}
For each real number $\eps >0$ there exists $c \in \bZ$ such that,
if $M$ is a triangle-free restriction of 
a binary projective geometry $G\cong \PG(r-1,2)$,
then there is a flat $F$ of $G$ such that
$r(F) \ge r(G)-c$ and $|F\cap E(M)|\le \eps 2^{r(F)}$.
\end{theorem}



\section{Regularity}

We will largely use the standard notation of matroid
theory [\ref{oxley}], but it will also be
convenient to think of a rank-$r$ binary matroid as a
subset of the vector space $V = \GF(2)^r$. This change is
purely notational; if $X \subseteq V$ then we write $M(X)$ for
the binary matroid on $X$ represented by a binary matrix with
column set $X$. If $0 \notin X$ then $M(X)$ is simple.
We define the {\em critical number} of $X$ to be the
critical number of $M(X)$; that is, the minimum codimension of a
subspace of $V$ disjoint from $X$.

Green used Fourier-analytic techniques to prove his regularity
lemma for abelian groups and to derive applications
in additive combinatorics; these techniques are discussed in greater 
detail in the book of Tao and Vu~[\ref{tv06}, Chapter 4]. 
Fortunately, although this theory has many technicalities,
the group $\GF(2)^n$ is among its simplest applications. 

Let $V = \GF(2)^r$ and let $X\subseteq V$.
Note that, if $H$ is a codimension-$1$ subspace of $V$,
then $|H| = |V \del H|$.
We say that $X$ is {\em $\eps$-uniform} if
for each codimension-$1$ subspace $H$ of $V$ we have
$$ | \,|H\cap X| - |X\del H|\, | \le \eps |V|.$$
In Lemma~\ref{sumlemma} we will see that,
for small $\eps$, the $\eps$-uniform sets are `pseudorandom'.

Let $H$ be a subspace of $V$.
For each $v \in V$, let $H_v(X) = \{h \in H: h + v \in X\}$.
For $\eps > 0$, we say $H$ is \emph{$\eps$-regular}
with respect to $V$ and $X$ if $H_v(X)$ is $\eps$-uniform
in $H$ for all but $\eps |V|$ values of $v \in V$.

Regularity captures the way that $X$ is distributed among
the cosets of $H$ in $V$.  For $v\in V$,
we let $X+v = \{x+v\, : \, x\in X\}$;
thus $X+v$ is a translation of $X$. Note that
$X+v$ is $\eps$-uniform if and only if $X$ is.
Also note that $H_v(X) + v = X\cap H'$ where $H'=H+v$
is the coset of $H$ in $V$ that contains $v$.
Therefore, if $u,v\in H'$, then
$H_u(X)$ and $H_v(X)$ are translates of one another.
So $H$ is $\eps$-regular if, for all but an 
$\eps$-fraction of cosets $H'$ of $H$,
the set $(H'\cap X)+v$ is $\eps$-uniform  in $H$ for some $v\in H'$.

The following result of Green [\ref{green}] guarantees a regular
subspace of bounded codimension. Here $W(t)$ denotes an
exponential tower of $2$'s of height $\lceil t \rceil$.

\begin{theorem}[Green's regularity lemma] \label{regularity}
Let $V = \GF(2)^n$, $X\subseteq V$, and let $\eps > 0$ be a real number.
Then there is  a
subspace $H$ of $v$ that is $\eps$-regular with respect to $X$
and $V$ and has codimension at most $W(\eps^{-3})$ in $V$. 
\end{theorem}

Let $A\subseteq V$ with $|A| =\alpha |V|$.
For $x \in V$ and $k \in \bZ$, we let
$S(A,k;x)$ denote the set of $k$-tuples in $A^k$ with sum
equal to $x$. Clearly $|S(A,k;x)| \le \alpha^{k-1}|V|^{k-1}$. If $A$ were
a random subset of $V$, we would expect around a
$|V|^{-1}$-fraction of the tuples in $A^k$ to
sum to $x$, which would give
$|S(A,k;x)| \approx \alpha^{k} |V|^{k-1}$; the next
lemma, a corollary of [\ref{tv06}, Lemma 4.13], bounds the error
in such an estimate when $A$ is uniform.
	
\begin{lemma}\label{sumlemma}
Let $V = \GF(2)^n$, let $x\in V$, and let
$A\subseteq V$ with $|A|=\alpha |V|$.
For each integer $k\ge 3$ and real $\eps>0$,
if $A$ is $\eps$-uniform, then
\[ 
\left|S(A,k;x) \right| \ge
 (\alpha^{k} - \eps^{k-2}) |V|^{k-1}. \]
\end{lemma}

Observe that, if $x\in A$ and $\{x,a_1,\ldots,a_{k-1}\}$
is a $k$-element circuit in $M(A)$ that contains $x$,
then $(a_1,\ldots,a_{k-1})\in S(A,k-1;x)$.
However the converse need not be true; if
$(a_1,\ldots,a_{k-1})\in S(A,k-1;x)$ then
$\{x,a_1,\ldots,a_{k-1}\}$ is a $k$-element circuit unless
some proper sub-tuple of $(a_1,\ldots,a_{k-1})$ 
sums to zero.  We let $S_0(A,k;x)$ denote
the set of $k$-tuples in $S(A,k;x)$ having some proper nonempty
sub-tuple with sum $0$. We argue that $S_0(A,k;x)$ is small. 

\begin{lemma}\label{degenerate}
Let $V = \GF(2)^n$, let $k$ be an integer, let $x\in V$, and let
$A\subseteq V$. Then
$|S_0(A,k;x)| \le 2^k|A|^{k-2}$. 
\end{lemma}

\begin{proof}
If some subtuple has sum $0$ then its complementary tuple
has sum $x$. Summing over all possible nonempty sub-tuples,
we have
\begin{eqnarray*}
|S_0(A,k;x)| &\le& \sum_{i = 1}^{k-1} \binom{k}{i} |S(A,i;0)||S(A,k-i;x)|\\
&\le& \sum_{i = 1}^{k-1} \binom{k}{i} |A|^{i-1}|A|^{k-i-1} \\
&\le& 2^k|A|^{k-2}. 
\end{eqnarray*}
\end{proof}

\section{The Main Result}

\begin{theorem}\label{maintech}
For every real number $\alpha >0$ and odd integer $k\ge 5$,
there exists a real number $\beta>0$ and integer  $c$ such
that, if $M$ is a simple binary matroid with
$|M| \ge \alpha 2^{r(M)}$, then either
$M$ has critical number at most $c$, or some element 
of $M$ is contained in at least $\beta 2^{(k-2)r(M)}$
distinct $k$-element circuits of $M$. 
\end{theorem}

\begin{proof}
Let $\alpha>0$ be real and let $k\ge 5$ be an odd integer.
Choose $\eps >0$ so that
\[(\alpha - \eps)^{k-1} - \eps^{k-3}  > 0,\]
let $\alpha_0 = \alpha - \eps$, and then choose $r_0\in\bZ$ so that
\[\alpha_0^{k-1} - \eps^{k-3} - 2^{k-1+W(\eps^{-3}) - r_0} > 0.\]
Let $s_0 = W(\eps^{-3})$, let $c = \max(r_0,s_0)$ and let 
\[ \beta = \frac{2^{(2-k)s_0}}{(k-1)!}\left(\alpha_0^{k-1}
-\eps^{k-3} - 2^{k-1+s_0-r_0}\right). \]
By our choice of $r_0$, we have $\beta > 0$.

Let $M$ be a simple rank-$r$ binary matroid with
$|M| \ge \alpha 2^r$.
Let $V=\GF(2)^r$ and let $X\subseteq V$ such that $M \cong M(X)$.
By Green's regularity lemma,
there is an $\eps$-regular subspace $H$ of $V$ with
codimension $s\le c$.

\begin{claim}
There is some $a \in V$ such that $H_a(X)$ is $\eps$-uniform in
$H$ and satisfies $|H_a(X)| \ge \alpha_0|H|$. 
\end{claim}

\begin{proof}[Proof of claim:]
Let $V_0$ be the set of $v \in V$ for which
$H_v(X)$ is not $\eps$-uniform; we have $|V_0| \le \eps |V|$
by regularity. In summing $|H_v(X)|$ over all $v \in V$,
we count each $x \in X$ with multiplicity $|H|$,
so 
$$\sum_{v \in V}|H_v(X)| = |X||H| \ge \alpha |V| |H|.$$
On the other hand, $\sum_{v \in V_0} |H_v(X)| \le \eps |V| |H|$.
Thus there exists an element $a \in V \del V_0$ with
$$ |H_a(X)| \ge \tfrac{\alpha |V||H| -
\eps |V||H|}{|V \del V_0|} \ge (\alpha - \eps)|H| = \alpha_0|H|,$$  
as required.
\end{proof}

Since $|H_a(X)|$ is constant as $a$ ranges over each coset of $H$,
we may choose $a = 0$ if $a \in H$. Let $A = H_a(X)$. 
We may assume that $M$ has critical number greater than $c$ and, hence,
there exists $x\in H\cap X$.

\begin{claim}\quad
$|S(A,k-1;x) \del S_0(A,k-1;x)| \ge \beta (k-1)!\  2^{(k-2)r}$.
\end{claim}

\begin{proof}[Proof of claim:]
By Lemma~\ref{sumlemma}, we have 
\begin{align*}
|S(A,k-1;x)| 
&\ge \left(\alpha_0^{k-1} - \eps^{k-3}\right) |H|^{k-2}\\
&= \left(\alpha_0^{k-1} - \eps^{k-3}\right)2^{(k-2)(r-s)}.
\end{align*}
By Lemma~\ref{degenerate} we have
\begin{align*}
|S_0(A,k-1;x)| &\le 2^{k-1}|A|^{k-3}\\
& \le 2^{k-1}|H|^{k-3}\\
& =  2^{k-1+s-r}2^{(k-2)(r-s)}
\end{align*}
Combining these and using $r \ge r_0$ and $s \le s_0$, the claim follows. 
\end{proof}

Let $w = (w_1, \dotsc, w_{k-1}) \in S(A,k-1;x) \del S_0(A,k-1;x)$.
The tuple $w' = (w_1 + a, w_2 + a, \dotsc, w_{k-1}+a,x)$
is contained in $X^{k}$, sums to zero, and since no
sub-tuple of $w$ sums to zero, the elements of $w'$ are
distinct and have no sub-tuple summing to zero . (If $a = 0$ this is clear, and otherwise $a \notin H$ so the $w_i + a$ are distinct from $x$.) Therefore $w'$
corresponds to a circuit of $M(X)$ containing $x$. Taking
into account permutations of $w$, it follows that
$x$ is in at least $\beta 2^{(k-2)r}$ distinct
$k$-element circuits of $M(X)$.
\end{proof}

\section{Triangle-free binary matroids}

Finally, to prove Theorem~\ref{weaker}, we need a variation
on Lemma~\ref{sumlemma}, also following from [\ref{tv06}, Lemma 4.13].
Let $V=\GF(2)^r$.  For sets $A_1,A_2,A_3 \subseteq V$,
let $T(A_1,A_2,A_3)$ be the set of triples in 
$A_1 \times A_2 \times A_3$ with sum zero. 

\begin{lemma}\label{triangles}
Let $V \in \GF(2)^n$ and $\eps > 0$.
Let $A_1,A_2,A_3 \subseteq V$ with $|A_i| = \alpha_i |V|$.
If $A_1$ is $\eps$-uniform, then 
\[|T(A_1,A_2,A_3)| \ge (\alpha_1\alpha_2\alpha_3 - \eps)|V|^2.\]
\end{lemma}

\begin{proof}[Proof of Theorem~\ref{weaker}]
Let $\eps > 0$. Let $\delta$ be a real number such
that $\eps(\eps-\delta)^2 > \delta > 0$, and let $c = W(\delta^{-3})$.

Let $M$ be a simple rank-$r$ triangle-free binary matroid.
If $|M| \le \eps 2^r$ then the theorem holds, so we may
assume for a contradiction that $|M| > \eps 2^r$. Let
$V = \GF(2)^r$ and $X \subseteq V$ be such that $M \cong M(X)$. 
	
By Green's regularity lemma there is an $\delta$-regular subspace
$H$ of $V$ with codimension at most $c$. As in the first
claim of the proof of Theorem~\ref{maintech}, there is some
$a \in Z$ such that $H_a(X)$ is $\delta$-regular and satisfies
$|H_a(X)| \ge \eps - \delta$. We may choose $a$ such that
either $a = 0$ or $a \notin H$. Let $A = H_a(X)$. 
	
If $|X \cap H| \le \eps |H|$, then the theorem holds.
Otherwise, by Lemma~\ref{triangles}, we have
$|T(A,A,X \cap H)| \ge (\eps(\eps - \delta)^2 - \delta)|H|^2 > 0$,
so there is some triple $(x,y,z)$ with $x+y+z = 0$,
where $x,y \in A$ and $z \in X \cap H$. Now
$\{x+a,y+a,z\}$ is a triangle of $M(X)$, a contradiction.  
\end{proof}

\section*{References}
\newcounter{refs}
\begin{list}{[\arabic{refs}]}
{\usecounter{refs}\setlength{\leftmargin}{10mm}\setlength{\itemsep}{0mm}}

\item\label{erdossimonovits}
P. Erd\H os, M. Simonovits,
On a valence problem in extremal graph theory,
Discrete Math. 5 (1973), 323-334.

\item\label{fk}
H. Furstenberg, Y. Katznelson,
IP-sets, Szemer\'edi's Theorem and Ramsey Theory,
Bull. Amer. Math. Soc. (N.S.) 14 no. 2 (1986), 275--278.

\item\label{green}
B. Green,
A Szemer\'{e}di-type regularity lemma in abelian groups, with applications,
Geometric \& Functional Analysis GAFA 15 (2005), 340--376.

\item \label{oxley}
J. G. Oxley, 
Matroid Theory,
Oxford University Press, New York (2011).

\item\label{tv06}
T. C. Tao and V. H. Vu, 
Additive Combinatorics, 
Cambridge Studies in Advanced Mathematics, 105, 
Cambridge University Press, Cambridge (2006). 

\item\label{t07}
C. Thomassen, 
On the chromatic number of pentagon-free graphs of large minimum degree, 
Combinatorica 27 (2007), 241--243.

\end{list}		
\end{document}